\documentclass[reqno,12pt]{amsart}
\usepackage{amscd,amsfonts,amssymb}
\numberwithin{equation}{section}
\newtheorem{theorem}{Theorem}[section]
\newtheorem{lemma}{Lemma}[section]

\theoremstyle{definition}

\theoremstyle{remark}

\author{A. A. Kon'kov}
\address{Department of Differential Equations,
Faculty of Mechanics and Mathematics,
Mo\-s\-cow Lo\-mo\-no\-sov State University,
Vorobyovy Gory,
119992 Moscow, Russia}
\email{konkov@mech.math.msu.su}
\title{On Blow-up Conditions for Solutions of Differential Inequalities with $\varphi$-Laplacian}
\keywords{Monlinear inequalities, $\varphi$-Laplacian, Unbounded domains}
\date{}
\begin{document}

\begin{abstract}
Let $\Omega \ne \emptyset$ be an unbounded open subset of ${\mathbb R}^n$, $n \ge 2$.
We obtain blow-up conditions for non-negative solutions of the problem
$$
    {\mathcal L}_\varphi u
    \ge
    F (x, u)
    \quad
    \mbox{in }
    \Omega,
    \quad
    \left.
        u
    \right|_{
        \partial \Omega
    }
    =
    0,
$$
where $\varphi$ and $F$ are some function and
$$
    {\mathcal L}_\varphi u
    =
    \operatorname{div}
    \left(
        \frac{
            \varphi (|\nabla u|)
        }{
            |\nabla u|
        }
        \nabla u
    \right)
$$
is the $\varphi$-Laplace operator.
\end{abstract}






\maketitle

\section{Introduction}

Let $\Omega \ne \emptyset$ be an unbounded open subset of ${\mathbb R}^n$.
We study solutions of the inequality
\begin{equation}
    {\mathcal L}_\varphi u
    \ge
    F (x, u)
    \quad
    \mbox{in }
    \Omega,
    \label{1.1}
\end{equation}
satisfying the condition
\begin{equation}
    \left.
        u
    \right|_{
        \partial \Omega
    }
    =
    0,
    \label{1.2}
\end{equation}
where $F$ is a non-negative function and
\begin{equation}
    {\mathcal L}_\varphi u
    =
    \operatorname{div}
    \left(
        \frac{
            \varphi (|\nabla u|)
        }{
            |\nabla u|
        }
        \nabla u
    \right)
    \label{1.3}
\end{equation}
is the $\varphi$-Laplace operator with some increasing one-to-one function
$\varphi : [0, \infty) \to [0, \infty)$
such that
\begin{equation}
    \left(
        \frac{
            \varphi (|\xi|)
        }{
            |\xi|
        }
        \xi
        -
        \frac{
            \varphi (|\zeta|)
        }{
            |\zeta|
        }
        \zeta
    \right)
    (\xi - \zeta)
    >
    0
    \label{1.4}
\end{equation}
for all $\xi, \zeta \in {\mathbb R}^n$, $\xi \ne \zeta$. As is
customary, by $B_r$ and $S_r$ we mean the open ball  and the sphere
in ${\mathbb R}^n$ of radius $r > 0$ centered at zero. Let us denote
$
    \Omega_{r, \sigma}
    =
    \Omega
    \cap
    B_{\sigma r}
    \setminus
    \overline{
        B_{r / \sigma}
    },
$ $r > 0$, $\sigma > 1$. Following~\cite{LU}, we say that $f \in
W_{p, loc}^1 (\Omega)$, $p \ge 1$, if $f \in W_p^1 (\Omega \cap
B_r)$ for any real number $r > 0$. The space $L_{p, loc} (\Omega)$,
$p \ge 1$, is defined in a similar way.

The phenomenon of the absence of solutions of differential
equations and inequalities also known as blow-up has been
intensively studied by many mathematicians~[1--6, 8--15, 17].
Among the inequalities with nonlinearity in the principal part,
inequalities containing the $\varphi$-Laplacian are most general.
This operator was introduced by J.~Serrin. It appears in various
applications, for instants, in the theory of Riemannian
manifolds~\cite{PRS}.

In the paper presented to your attention, we manage to  obtain
sufficient blow-up conditions for solutions of inequalities with the
$\varphi$-Laplacian and a nonlinearity of a very general form in the
right-hand side. The exactness of these conditions is shown in
Examples~1 and~2.

A function $u \in W_{1, loc}^1 (\Omega) \cap L_{\infty, loc}
(\Omega)$ is called a weak solution of~\eqref{1.1} if $
    F (x, u) \in L_{1, loc} (\Omega),
$
$
    \varphi (|\nabla u|) |\nabla u| \in L_{1, loc} (\Omega),
$
and
$$
    -
    \int_\Omega
    \frac{
        \varphi (|\nabla u|)
    }{
        |\nabla u|
    }
    \nabla u \nabla \psi
    \,
    dx
    \ge
    \int_\Omega
    F (x, u)
    \psi
    \,
    dx
$$
for any non-negative function
$
    \psi
    \in
    {
        \stackrel{\rm \scriptscriptstyle o}{W}\!\!{}_1^1
        (
            \Omega
        )
    }
    \cap
    L_\infty (\Omega)
$
with a compact support such that
\begin{equation}
    \int_\Omega
    \varphi (|\nabla u|)
    |\nabla \psi|
    \,
    dx
    <
    \infty.
    \label{1.6}
\end{equation}
Analogously, a function $u \in W_{1, loc}^1 (\Omega)  \cap
L_{\infty, loc} (\Omega)$ is a weak solution of the inequality
$$
    {\mathcal L}_\varphi u
    \le
    F (x, u)
    \quad
    \mbox{in }
    \Omega
$$
if
$
    F (x, u) \in L_{1, loc} (\Omega),
$
$
    \varphi (|\nabla u|) |\nabla u| \in L_{1, loc} (\Omega),
$
and
$$
    -
    \int_\Omega
    \frac{
        \varphi (|\nabla u|)
    }{
        |\nabla u|
    }
    \nabla u \nabla \psi
    \,
    dx
    \le
    \int_\Omega
    F (x, u)
    \psi
    \,
    dx
$$
for any non-negative function
$
    \psi
    \in
    {
        \stackrel{\rm \scriptscriptstyle o}{W}\!\!{}_1^1
        (
            \Omega
        )
    }
    \cap
    L_\infty (\Omega)
$
with a compact support such that~\eqref{1.6} holds.
In its turn, a function $u \in W_{1, loc}^1 (\Omega)  \cap
L_{\infty, loc} (\Omega)$ is called a weak solution of the equation
$$
    {\mathcal L}_\varphi u
    =
    F (x, u)
    \quad
    \mbox{in }
    \Omega
$$
if
$
    F (x, u) \in L_{1, loc} (\Omega),
$
$
    \varphi (|\nabla u|) |\nabla u| \in L_{1, loc} (\Omega),
$
and
$$
    -
    \int_\Omega
    \frac{
        \varphi (|\nabla u|)
    }{
        |\nabla u|
    }
    \nabla u \nabla \psi
    \,
    dx
    =
    \int_\Omega
    F (x, u)
    \psi
    \,
    dx
$$
for any function
$
    \psi
    \in
    {
        \stackrel{\rm \scriptscriptstyle o}{W}\!\!{}_1^1
        (
            \Omega
        )
    }
    \cap
    L_\infty (\Omega)
$
with a compact support satisfying relation~\eqref{1.6}.

Condition~\eqref{1.2} means that
$
    \gamma u
    \in
    {
        \stackrel{\rm \scriptscriptstyle o}{W}\!\!{}_1^1
        (
            \Omega
        )
    }
$
for all $\gamma \in C_0^\infty ({\mathbb R}^n)$.
Note that, for $\Omega = {\mathbb R}^n$, this condition is obviously valid.

\medskip

\noindent {\bf Remark 1.} We understand the support of the function
$\psi$ in the sense of distributions from ${\mathcal D}' ({\mathbb
R}^n)$. In so doing, we extend $\psi$ by zero outside the set
$\Omega$. The inclusion $
    \varphi (|\nabla u|) |\nabla u| \in L_{1, loc} (\Omega),
$
obviously, implies that
$
    \varphi (|\nabla u|) \in L_{1, loc} (\Omega).
$
Indeed, for any real number $r > 0$ we have
$$
    \int_{\Omega \cap B_r}
    \varphi (|\nabla u|)
    \,
    dx
    =
    \int_{
        \{ x \in \Omega \cap B_r : |\nabla u (x)| \le 1 \}
    }
    \varphi (|\nabla u|)
    \,
    dx
    +
    \int_{
        \{ x \in \Omega \cap B_r : |\nabla u (x)| > 1 \}
    }
    \varphi (|\nabla u|)
    \,
    dx,
$$
where
$$
    \int_{
        \{ x \in \Omega \cap B_r : |\nabla u (x)| \le 1 \}
    }
    \varphi (|\nabla u|)
    \,
    dx
    \le
    \varphi (1)
    \operatorname{mes} \Omega \cap B_r
    <
    \infty
$$
and
$$
    \int_{
        \{ x \in \Omega \cap B_r : |\nabla u (x)| > 1 \}
    }
    \varphi (|\nabla u|)
    \,
    dx
    \le
    \int_{
        \{ x \in \Omega \cap B_r : |\nabla u (x)| > 1 \}
    }
    \varphi (|\nabla u|)
    |\nabla u|
    \,
    dx
    <
    \infty.
$$

\bigskip

A function $g : (0, \infty) \to (0, \infty)$ is called semi-multiplicative if
$
    g (t_1) g (t_2) \ge g (t_1 t_2)
$
for all $t_1, t_2 \in (0, \infty)$. In the case where the inequality
in the last formula  can be replaced by equality, the function $g$
is multiplicative. Using the isomorphism $t \mapsto e^t$ between the
additive group of real numbers and the multiplicative group of
positive reals, one can assert that a function $g$ is
semi-multiplicative if and only if $G (t) = \ln g (e^t)$ is
semi-additive, i.e.,
$$
    G (t_1) + G (t_2) \ge G (t_1 t_2)
$$
for all $t_1, t_2 \in {\mathbb R}$.
In so doing, a function $g : (0, \infty) \to (0, \infty)$ is almost semi-multiplicative if
\begin{equation}
    g (t_1) g (t_2) \ge c g (t_1 t_2)
    \label{1.7}
\end{equation}
with some constant $c > 0$ for all $t_1, t_2 \in (0, \infty)$.
In its turn, $G : {\mathbb R} \to {\mathbb R}$ is almost semi-additive function if
$
    G (t_1) + G (t_2) \ge G (t_1 t_2) - C
$
with some constant $C > 0$ for all $t_1, t_2 \in {\mathbb R}$. It
can easily be seen that a function $g$ is almost
semi-multiplicative if and only if $G (t) = \ln g (e^t)$ is almost
semi-additive.

The simplest example of an almost semi-multiplicative  function is
the power function $g (t) = t^\lambda$, $\lambda \in {\mathbb R}$.
This function is also multiplicative. As another example, we can
take $g (t) = t^\lambda \ln^\nu (2 + t)$, where $\lambda \ge 0$ and
$\nu \ge 0$ are real numbers.

Note that the product and superposition of almost
semi-multiplicative functions is also an almost semi-multiplicative
function. Taking $t_1 = \tau_1 \tau_2$ and $t_2 = 1 / \tau_2$
in~\eqref{1.7}, we obtain
\begin{equation}
    g (\tau_1 \tau_2)
    \ge
    \frac{
        c
        g (\tau_1)
    }{
        g (1 / \tau_2)
    }
    \label{1.8}
\end{equation}
for all real numbers $\tau_1 > 0$ and $\tau_2 > 0$.  According
to~\eqref{1.7} and~\eqref{1.8}, for any almost semi-multiplicative
function $g$ and real number $c > 0$ there exist constants $c_1 > 0$
and $c_2 > 0$ such that
$
    c_1 g (t)
    \le
    g (c t)
    \le
    c_2 g (t)
$
for all real numbers $t > 0$.

Below we assume that
$\Omega \cap S_r \ne \emptyset$ for all $r \ge r_0$, where $r_0 > 0$ is a real number.
In addition, let
\begin{equation}
    \inf_{
        x \in \Omega_{r, \sigma}
    }
    F (x, t)
    \ge
    f (r)
    g (t)
    \label{1.5}
\end{equation}
with some real number $\sigma > 1$
for all $r \ge r_0$ and $t > 0$, where
$f \in L_{\infty, loc} ([r_0, \infty))$
and
$g\in L_{\infty, loc} (0, \infty)$
are non-negative functions with
$
    \inf_K g > 0
$
for any compact set $K \subset (0, \infty)$. We also assume that the
function $g$ in~\eqref{1.5} and the  function inverse to $\varphi$
in~\eqref{1.3} are almost semi-multiplicative.

Denote
$$
    \eta (t)
    =
    \frac{
        t
    }{
        \varphi^{-1}
        \left(
            \frac{1}{t}
        \right)
    },
    \quad
    t > 0.
$$
It is easy to see that $\eta$ is an increasing  one-to-one mapping
of the set $(0, \infty)$ into itself. In so doing, since
$\varphi^{-1}$ is an almost semi-multiplicative function, for any
real number $c > 0$ there exist constants $c_1 > 0$ and $c_2 > 0$
such that
\begin{equation}
    c_1 \eta (t)
    \le
    \eta (c t)
    \le
    c_2 \eta (t)
    \label{1.10}
\end{equation}
and
\begin{equation}
    c_1 \eta^{-1} (t)
    \le
    \eta^{-1} (c t)
    \le
    c_2 \eta^{-1} (t)
    \label{1.11}
\end{equation}
for all real numbers $t > 0$.

\section{Main results}

\begin{theorem}\label{t2.1}
Let $\eta$ be a convex function and, moreover,
\begin{equation}
    \int_1^\infty
    \eta^{-1}
    \left(
        \frac{
            t
        }{
            \varphi^{-1} (g (t))
        }
    \right)
    \frac{dt}{t}
    <
    \infty
    \label{t2.1.1}
\end{equation}
and
\begin{equation}
    \int_{r_0}^\infty
    \eta
    \left(
        \frac{
            r
        }{
            \eta^{-1}
            \left(
                \varphi^{-1}
                \left(
                    \frac{1}{f (r)}
                \right)
            \right)
        }
    \right)
    \frac{dr}{r}
    =
    \infty.
    \label{t2.1.2}
\end{equation}
Then, any non-negative weak solution of~\eqref{1.1}, \eqref{1.2} is
identically equal to zero.
\end{theorem}

\begin{theorem}\label{t2.2}
Let condition~\eqref{t2.1.1} be valid and, moreover,
\begin{equation}
    \int_1^\infty
    \frac{
        dt
    }{
        \varphi^{-1} (g (t))
    }
    <
    \infty
    \label{t2.2.1}
\end{equation}
and
\begin{equation}
    \int_{r_0}^\infty
    \min
    \left\{
        \frac{
            \eta (r)
        }{
            r
            \varphi^{-1}
            \left(
                \frac{
                    1
                }{
                    f (r)
                }
            \right)
        },
        \frac{
            1
        }{
            \eta^{-1}
            \left(
                \varphi^{-1}
                \left(
                    \frac{
                        1
                    }{
                        f (r)
                    }
                \right)
            \right)
        }
    \right\}
    dr
    =
    \infty.
    \label{t2.2.2}
\end{equation}
Then, any non-negative weak solution of~\eqref{1.1}, \eqref{1.2} is
identically equal to zero.
\end{theorem}

\noindent {\bf Remark 2.} In the case of $f (r) = 0$, we assume by
definition that the integrands in the left-hand side
of~\eqref{t2.1.2} and~\eqref{t2.2.2} are equal to zero.


Theorems~\ref{t2.1} and~\ref{t2.2} are proved in Section~\ref{proofs}.
Now we demonstrate their exactness.

\medskip

\noindent {\bf Example 1.} Consider the inequality
\begin{equation}
    \Delta_p u \ge c (x) |u|^\lambda
    \quad
    \mbox{in } {\mathbb R}^n,
    \label{e2.1.1}
\end{equation}
where
$$
    \Delta_p u = \operatorname{div} (|\nabla u|^{p - 2} \nabla u),
    \quad
    p > 1,
$$
is the p-Laplacian and $c \in L_{\infty, loc} ({\mathbb R}^n)$ is a
non-negative function satisfying the relation
\begin{equation}
    c (x) \sim |x|^s
    \quad
    \mbox{as } x \to \infty,
    \label{e2.1.2}
\end{equation}
i.e.,
$
    c_1 |x|^s \le c (x) \le c_2 |x|^s
$
with some constants $c_1 > 0$ and $c_2 > 0$ for all $x$ in a neighborhood of infinity.

By Theorem~\ref{t2.1}, if
\begin{equation}
    \lambda > p - 1
    \quad
    \mbox{and}
    \quad
    s \ge - p,
    \label{e2.1.3}
\end{equation}
then any non-negative weak solution of~\eqref{e2.1.1} is identically equal to zero.
It can be shown that condition~\eqref{e2.1.3} is exact~\cite{meIM,meDM}.

\medskip

\noindent {\bf Example 2.} We examine the critical exponent $\lambda
= p - 1$ in~\eqref{e2.1.3}. Namely, consider the inequality
\begin{equation}
    {\mathcal L}_\varphi u
    \ge
    c (x) |u|^{p - 1}
    \quad
    \mbox{in } {\mathbb R}^n,
    \label{e2.2.1}
\end{equation}
where the increasing one-to-one function
$\varphi : [0, \infty) \to [0, \infty)$
such that
\begin{equation}
    \varphi (t)
    \sim
    t^{p - 1}
    \log^\nu t
    \quad
    \mbox{as } t \to \infty
    \quad
    \mbox{and}
    \quad
    \varphi (t)
    \sim
    t^{p - 1}
    \log^{- \nu} \frac{1}{t}
    \quad
    \mbox{as } t \to +0
    \label{e2.2.2}
\end{equation}
and $c \in L_{\infty, loc} ({\mathbb R}^n)$ is a  non-negative
function for which~\eqref{e2.1.2} holds. It is not difficult to
verify that the following relations are valid
$$
    \varphi^{-1} (t)
    \sim
    t^{1 / (p - 1)}
    \log^{- \nu / (p - 1)}
    t
    \quad
    \mbox{as } t \to \infty,
    \qquad
    \varphi^{-1} (t)
    \sim
    t^{1 / (p - 1)}
    \log^{\nu / (p - 1)}
    \frac{1}{t}
    \quad
    \mbox{as } t \to +0,
$$
$$
    \eta (t)
    \sim
    t^{p / (p - 1)}
    \log^{- \nu / (p - 1)}
    t
    \quad
    \mbox{as } t \to \infty,
    \qquad
    \eta (t)
    \sim
    t^{p / (p - 1)}
    \log^{\nu / (p - 1)}
    \frac{1}{t}
    \quad
    \mbox{as } t \to +0,
$$
$$
    \eta^{-1} (t)
    \sim
    t^{(p - 1) / p}
    \log^{\nu / p}
    t
    \quad
    \mbox{as } t \to \infty,
    \qquad
    \eta^{-1} (t)
    \sim
    t^{(p - 1) / p}
    \log^{- \nu / p}
    \frac{1}{t}
    \quad
    \mbox{as } t \to +0.
$$
Thus, applying Theorem~\ref{t2.2}, we obtain that, for
\begin{equation}
    \nu < -p
    \quad
    \mbox{and}
    \quad
    s \ge -p,
    \label{e2.2.3}
\end{equation}
any non-negative weak solution of~\eqref{e2.2.1} is identically equal to zero.
In so doing, if $\nu \ge -p$ and
\begin{equation}
    \varphi' (t)
    \sim
    t^{p - 2}
    \log^\nu t
    \quad
    \mbox{as } t \to \infty
    \label{e2.2.4}
\end{equation}
in addition to~\eqref{e2.2.2}, then for any $s \in {\mathbb R}$ the function
$$
    u (x)
    =
    \left\{
        \begin{aligned}
            &
            e^{
                e^{
                    |x|^{
                        (p + |s|) / p
                    }
                }
            }
            -
            e^{
                e^{
                    r_0^{
                        (p + |s|) / p
                    }
                }
            },
            &
            &
            |x| > r_0,
            \\
            &
            0,
            &
            &
            |x| \le r_0,
        \end{aligned}
    \right.
$$
where $r_0 > 0$ is large enough, is a weak solution
of~\eqref{e2.2.1} with some non-negative function $c \in L_{\infty,
loc} ({\mathbb R}^n)$, satisfying~\eqref{e2.1.2}. Hence, the first
inequality in~\eqref{e2.2.3} is exact.

Now, let $\nu < -p$ and $s < -p$. Then, assuming that~\eqref{e2.2.4}
holds  in addition to~\eqref{e2.2.2}, we obtain that
$$
    u (x)
    =
    \left\{
        \begin{aligned}
            &
            e^{
                |x|^{
                    (s + p) / (p + \nu)
                }
            }
            -
            e^{
                r_0^{
                    (s + p) / (p + \nu)
                }
            },
            &
            &
            |x| > r_0,
            \\
            &
            0,
            &
            &
            |x| \le r_0,
        \end{aligned}
    \right.
$$
where $r_0 > 0$ is large enough, is a weak solution
of~\eqref{e2.2.1}  with some non-negative function $c \in L_{\infty,
loc} ({\mathbb R}^n)$ satisfying~\eqref{e2.1.2}. Consequently, the
second inequality in~\eqref{e2.2.3} is also exact.

\section{Proof of Theorems~\ref{t2.1} and~\ref{t2.2}}\label{proofs}

We say that
\begin{equation}
    u_1
    \le
    u_2
    \quad
    \mbox{on } \partial \omega,
    \label{3.1}
\end{equation}
where $u_i \in W_{1, loc}^1 (\omega) \cap L_{\infty, loc} (\omega)$,
$i = 1, 2$, are some functions and $\omega$ is an open subset of
${\mathbb R}^n$, if $
    \gamma
    \max \{ u_1 - u_2, 0 \}
    \in
    {
        \stackrel{\rm \scriptscriptstyle o}{W}\!\!{}_1^1
        (
            \omega
        )
    }
$ for all $\gamma \in C_0^\infty ({\mathbb R}^n)$. If $\omega$ is a
bounded open subset of ${\mathbb R}^n$,  then condition~\eqref{3.1}
obviously means that $
    \max \{ u_1 - u_2, 0 \}
    \in
    {
        \stackrel{\rm \scriptscriptstyle o}{W}\!\!{}_1^1
        (
            \omega
        )
    }.
$

\begin{lemma}\label{l3.1}
Let $u_i \in W_1^1 (\omega) \cap L_\infty (\omega)$, $i = 1, 2$, be
weak solutions of the inequalities
\begin{equation}
    {\mathcal L}_\varphi u_1
    \ge
    \rho (x, u_1)
    \quad
    \mbox{in }
    \omega
    \label{l3.1.1}
\end{equation}
and
\begin{equation}
    {\mathcal L}_\varphi u_2
    \le
    \rho (x, u_2)
    \quad
    \mbox{in }
    \omega,
    \label{l3.1.2}
\end{equation}
satisfying condition~\eqref{3.1}, where $\omega$ is an open bounded
subset of ${\mathbb R}^n$  and $\rho$ is some function
non-decreasing with respect to the last argument. If $
    \varphi (|\nabla u_i|) |\nabla u_j| \in L_1 (\omega)
$
and
$
    \rho (x, u_i) u_j  \in L_1 (\omega),
$ $i,j = 1, 2$, then
\begin{equation}
    u_1 \le u_2
    \quad
    \mbox{a.e. in } \omega.
    \label{l3.1.3}
\end{equation}
\end{lemma}

\begin{proof}
From~\eqref{l3.1.1} and~\eqref{l3.1.2}, it follows that
\begin{equation}
    -
    \int_{\tilde \omega}
    \frac{
        \varphi (|\nabla u_1|)
    }{
        |\nabla u_1|
    }
    \nabla u_1
    \nabla \psi
    \,
    dx
    \ge
    \int_{\tilde \omega}
    \rho (x, u_1) \psi
    \,
    dx
    \label{pl3.1.1}
\end{equation}
and
\begin{equation}
    -
    \int_{\tilde \omega}
    \frac{
        \varphi (|\nabla u_2|)
    }{
        |\nabla u_2|
    }
    \nabla u_2
    \nabla \psi
    \,
    dx
    \le
    \int_{\tilde \omega}
    \rho (x, u_2) \psi
    \,
    dx,
    \label{pl3.1.2}
\end{equation}
where
$
    \psi
    =
    \max \{ u_1 - u_2, 0 \}
$
and $\tilde \omega = \{ x \in \omega : u_1 (x) > u_2 (x) \}$.
Thus, subtracting~\eqref{pl3.1.1} from~\eqref{pl3.1.2}, we obtain
$$
    \int_{\tilde \omega}
    \left(
        \frac{
            \varphi (|\nabla u_1|)
        }{
            |\nabla u_1|
        }
        \nabla u_1
        -
        \frac{
            \varphi (|\nabla u_2|)
        }{
            |\nabla u_2|
        }
        \nabla u_2
    \right)
    (\nabla u_1 - \nabla u_2)
    \,
    dx
    \le
    0.
$$
In view of~\eqref{1.4}, this readily implies~\eqref{l3.1.3}.
\end{proof}

We need the following generalization  of well-known Kato
theorem~\cite{Kato}.

\begin{lemma}\label{l3.2}
Let $u \in W_{1, loc}^1 (\omega) \cap L_{\infty, loc} (\omega)$ be a weak solution of the problem
\begin{equation}
    {\mathcal L}_\varphi u
    \ge
    a (x)
    \quad
    \mbox{in }
    \omega,
    \quad
    u
    \le
    0
    \quad
    \mbox{on } \partial \omega,
    \label{l3.2.1}
\end{equation}
where $\omega$ is an open subset  of ${\mathbb R}^n$ and $a \in
L_{1, loc} (\omega)$. We denote $\tilde \omega = \{ x \in \omega : u
(x) > 0 \}$ and
$$
    \tilde a (x)
    =
    \left\{
        \begin{aligned}
            &
            a (x)
            &
            &
            x \in \tilde \omega,
            \\
            &
            0,
            &
            &
            x \in {\mathbb R}^n \setminus \tilde \omega.
        \end{aligned}
    \right.
$$
Then, the function
$$
    \tilde u (x)
    =
    \left\{
        \begin{aligned}
            &
            u (x),
            &
            &
            x \in \tilde \omega,
            \\
            &
            0,
            &
            &
            x \in {\mathbb R}^n \setminus \tilde \omega,
        \end{aligned}
    \right.
$$
is a weak solution of the inequality
$
    {\mathcal L}_\varphi \tilde u \ge \tilde a (x)
    \quad
    \mbox{in } {\mathbb R}^n.
$
\end{lemma}

\begin{proof}
The proof, with minor differences,  repeats the reasoning given in~\cite[Lemma~4.2]{meJMAA2007}. 
In order not to be
unfounded, we present this proof in full. Let us take a
non-decreasing function  $\zeta \in C^\infty ({\mathbb R})$ such
that $\zeta = 0$ on the interval  $(-\infty, 1/4]$ and $\zeta = 1$
on $[3/4, \infty)$. Put $
    \zeta_\varepsilon (t)
    =
    \zeta (t / \varepsilon),
$
where $\varepsilon > 0$ is a real number.
In view of~\eqref{l3.2.1}, we have
\begin{equation}
    -
    \int_\omega
    \frac{
        \varphi (|\nabla u|)
    }{
        |\nabla u|
    }
    \nabla u \nabla (\zeta_\varepsilon (u) \psi)
    \,
    dx
    \ge
    \int_\omega
    a (x)
    \zeta_\varepsilon (u) \psi
    \,
    dx
    \label{pl3.2.1}
\end{equation}
for any non-negative function
$
    \psi
    \in
    {
        \stackrel{\rm \scriptscriptstyle o}{W}\!\!{}_1^1
        (
            \Omega
        )
    }
    \cap
    L_{\infty, loc} (\omega)
$
with a compact support such that
$$
    \int_\omega
    \varphi (|\nabla u|)
    |\nabla \psi|
    \,
    dx
    <
    \infty.
$$
Note that
$
    \varphi (|\nabla u|) |\nabla u| \in L_{1, loc} (\Omega)
$ by the definition of solutions of~\eqref{l3.2.1}; therefore,
$$
    \int_\omega
    \varphi (|\nabla u|)
    |\nabla (\zeta_\varepsilon (u) \psi)|
    \,
    dx
    <
    \infty.
$$
We also have
$
    \zeta_\varepsilon (u) \psi
    \in
    {
        \stackrel{\rm \scriptscriptstyle o}{W}\!\!{}_1^1
        (
            \omega
        )
    }
    \cap
    L_{\infty, loc} (\omega).
$
Hence, $\zeta_\varepsilon (u) \psi$ can be used as a test function in~\eqref{pl3.2.1}.

Expanding the brackets in the left-hand side of~\eqref{pl3.2.1}, we obtain
\begin{equation}
    -
    \int_\omega
    \zeta_\varepsilon (u)
    \frac{
        \varphi (|\nabla u|)
    }{
        |\nabla u|
    }
    \nabla u \nabla \psi
    \,
    dx
    \ge
    \int_\omega
    \varphi (|\nabla u|) |\nabla u|
    \zeta_\varepsilon' (u) \psi
    \,
    dx
    +
    \int_\omega
    a (x)
    \zeta_\varepsilon (u) \psi
    \,
    dx.
    \label{pl3.2.2}
\end{equation}
The first summand on the right in the last expression is non-negative, whereas
$$
    \int_\omega
    \zeta_\varepsilon (u)
    \frac{
        \varphi (|\nabla u|)
    }{
        |\nabla u|
    }
    \nabla u \nabla \psi
    \,
    dx
    \to
    \int_{
        {\mathbb R}^n
    }
    \frac{
        \varphi (|\nabla \tilde u|)
    }{
        |\nabla \tilde u|
    }
    \nabla \tilde u \nabla \psi
    \,
    dx
$$
and
$$
    \int_\omega
    a (x)
    \zeta_\varepsilon (u) \psi
    \,
    dx
    \to
    \int_{
        {\mathbb R}^n
    }
    \tilde a (x)
    \psi
    \,
    dx
$$
as $\varepsilon \to +0$. Thus, passing  to the limit
in~\eqref{pl3.2.2} as $\varepsilon \to +0$, we arrive at the
inequality
$$
    \int_{
        {\mathbb R}^n
    }
    \frac{
        \varphi (|\nabla \tilde u|)
    }{
        |\nabla \tilde u|
    }
    \nabla \tilde u \nabla \psi
    \,
    dx
    \ge
    \int_{
        {\mathbb R}^n
    }
    \tilde a (x)
    \psi
    \,
    dx.
$$
The proof is completed.
\end{proof}

From now on, we assume that $u$ is a non-negative  weak solution
of~\eqref{1.1}, \eqref{1.2}. Let us put
$$
    M (r)
    =
    \operatorname*{ess\,sup}\limits_{
        \Omega \cap B_r
    }
    u,
    \quad
    r > r_0.
$$
We obviously have $M (r) = M (r - 0)$ for all $r \in (r_0, \infty)$.
If $M$ is identically zero on $(r_0, \infty)$,  then $u$ is
identically zero a.e. in $\Omega$; therefore, it can be assumed that
$M (r_*) > 0$ for some $r_* \in (r_0, \infty)$.

Below, by $C$ we mean various positive constants  that can depend
only on $\sigma$, $g$, $\varphi$, $n$, and $M (r_*)$.

\begin{lemma}\label{l3.3}
Let $r_* \le r_1 < r_2$ be real numbers such that  $r_2 \le \sigma
r_1$ and $M (r_2) \le 2 M (r_1)$. Then,
\begin{equation}
    M (r_2) - M (r_1)
    \ge
    C
    (r_2 - r_1)
    \varphi^{-1}
    (
        (r_2 - r_1)
        {\mathcal F}
    ),
    \label{l3.3.1}
\end{equation}
where
\begin{equation}
    {\mathcal F}
    =
    \inf_{
        x \in \Omega \cap B_{r_2} \setminus B_{r_1},
        \;
        t \in [M (r_1) / 2, M (r_2)]
    }
    F (x, t).
    \label{l3.3.2}
\end{equation}
\end{lemma}

\begin{proof}
By Lemma~\ref{l3.2}, we have
$
    {\mathcal L}_\varphi \tilde u
    \ge
    \tilde F (x)
    \quad
    \mbox{in } {\mathbb R}^n,
$
where
$$
    \tilde u (x)
    =
    \left\{
        \begin{aligned}
            &
            u (x),
            &
            &
            x \in \Omega,
            \\
            &
            0,
            &
            &
            x \in {\mathbb R}^n \setminus \Omega,
        \end{aligned}
    \right.\quad {\text and}\quad
    \tilde F (x)
    =
    \left\{
        \begin{aligned}
            &
            F (x, u (x))
            &
            &
            x \in \Omega,
            \\
            &
            0,
            &
            &
            x \in {\mathbb R}^n \setminus \Omega.
        \end{aligned}
    \right.
$$
This obviously implies that
$
    {\mathcal L}_\varphi \tilde u
    \ge
    \rho (x, \tilde u)
    \quad
    \mbox{in } B_{r_2},
$
where
$$
    \rho (x, t)
    =
    \left\{
        \begin{aligned}
            &
            {\mathcal F},
            &
            &
            x \in B_{r_2} \setminus B_{r_1}
            \;
            \&
            \;
            t \ge M (r_1) / 2,
            \\
            &
            0,
            &
            &
            x \not\in B_{r_2} \setminus B_{r_1}
            \;
            \vee
            \;
            t < M (r_1) / 2,
        \end{aligned}
    \right.
$$

Let us further denote
$$
    w (r)
    =
    \int_{r_1}^r
    \varphi^{-1}
    \left(
        \frac{
            \tau^n - r_1^n
        }{
            n \tau^{n - 1}
        }
        {\mathcal F}
    \right)
    d\tau.
$$
By direct differentiation, it can be verified that $w$ satisfies the Cauchy problem
\begin{equation}
    \frac{
        1
    }{
        r^{n - 1}
    }
    \frac{d}{dr}
    \left(
        r^{n - 1}
        \varphi
        \left(
            \frac{d w}{d r}
        \right)
    \right)
    =
    {\mathcal F},
    \quad
    w (r_1)
    =
    \frac{d w}{d r} (r_1) = 0,
    \label{pl3.3.1}
\end{equation}
for the radial component of the operator $\Delta_\varphi$.
Let us shown that
\begin{equation}
    w (r_2)
    \ge
    C
    (r_2 - r_1)
    \varphi^{-1}
    (
        (r_2 - r_1)
        {\mathcal F}
    ).
    \label{pl3.3.2}
\end{equation}
Really, in accordance with Lagrange's theorem
$\tau^n - r_1^n = n \tau_*^{n - 1} (\tau - r_1)$
for some $\tau_* \in (r_1, \tau)$; therefore,
$$
    \frac{
        \tau^n - r_1^n
    }{
        n \tau^{n - 1}
    }
    \ge
    \frac{
        \tau - r_1
    }{
        \sigma^{n-1}
    }
$$
for all $\tau \in (r_1, r_2)$, whence it follows that
$$
    w (r_2)
    \ge
    \int_{r_1}^{r_2}
    \varphi^{-1}
    \left(
        \frac{
            \tau - r_1
        }{
            \sigma^{n-1}
        }
        {\mathcal F}
    \right)
    d\tau
    \ge
    \int_{
        (r_1 + r_2) / 2
    }^{
        r_2
    }
    \varphi^{-1}
    \left(
        \frac{
            \tau - r_1
        }{
            \sigma^{n-1}
        }
        {\mathcal F}
    \right)
    d\tau
    \ge
    \frac{r_2 - r_1}{2}
    \varphi^{-1}
    \left(
        \frac{
            r_2 - r_1
        }{
            2 \sigma^{n-1}
        }
        {\mathcal F}
    \right).
$$
Thus, to establish the validity of~\eqref{pl3.3.2}, it remains to note that
$$
    \varphi^{-1}
    \left(
        \frac{
            r_2 - r_1
        }{
            2 \sigma^{n-1}
        }
        {\mathcal F}
    \right)
    \ge
    C
    \varphi^{-1}
    (
        (r_2 - r_1)
        {\mathcal F}
    )
$$
in accordance with property~\eqref{1.8} of almost semi-multiplicative functions.

We put
$$
    v (x)
    =
    \tilde w (x)
    +
    \operatorname*{ess\,sup}\limits_{
        B_{r_2}
    }
    \,
    (\tilde u - \tilde w),
$$
where
$$
    \tilde w (x)
    =
    \left\{
        \begin{aligned}
            &
            w (|x|),
            &
            &
            x \in B_{r_2} \setminus B_{r_1},
            \\
            &
            0,
            &
            &
            x \in B_{r_1}.
        \end{aligned}
    \right.
$$
It can easily be see that
\begin{equation}
    v \ge \tilde u
    \quad
    \mbox{a.e. in } B_{r_2}
    \label{pl3.3.3}.
\end{equation}
Really, we have
$$
    \operatorname*{ess\,sup}\limits_{
        B_{r_2}
    }
    \,
    (\tilde u - \tilde w)
    \ge
    \tilde u (x) - \tilde w (x)
$$
for almost all $x \in B_{r_2}$. Hence,
$$
    v (x)
    =
    \tilde w (x)
    +
    \operatorname*{ess\,sup}\limits_{
        B_{r_2}
    }
    \,
    (\tilde u - \tilde w)
    \ge
    \tilde w (x)
    +
    \tilde u (x) - \tilde w (x)
    =
    \tilde u (x)
$$
for almost all $x \in B_{r_2}$.
Since $\tilde w$ is a non-negative function, it is also obvious that
\begin{equation}
    v (x)
    \ge
    \operatorname*{ess\,sup}\limits_{
        B_{r_2}
    }
    \,
    (\tilde u - \tilde w)
    \ge
    \operatorname*{ess\,sup}\limits_{
        B_{r_1}
    }
    \,
    (\tilde u - \tilde w)
    =
    M (r_1)
    \label{pl3.3.4}
\end{equation}
for all $x \in B_{r_2}$.

Let us establish the validity of the equality
\begin{equation}
    \left.
        v
    \right|_{
        S_{r_2}
    }
    =
    M (r_2).
    \label{pl3.3.5}
\end{equation}
Indeed, taking into account~\eqref{pl3.3.3}, we obtain
$$
    \left.
        v
    \right|_{
        S_{r_2}
    }
    =
    \sup_{
        B_{r_2}
    }
    v
    \ge
    \operatorname*{ess\,sup}\limits_{
        B_{r_2}
    }
    \tilde u
    =
    M (r_2);
$$
therefore, if~\eqref{pl3.3.5} is not valid, then
$
    \left.
        (v - \varepsilon)
    \right|_{
        S_{r_2}
    }
    >
    M (r_2)
$
for some real number $0 < \varepsilon < M (r_1) / 2$. In view
of~\eqref{pl3.3.1}, the function $\tilde  v = v - \varepsilon$ is a
weak solution of the equation
$$
    \Delta_\varphi \tilde v
    =
    \chi_{
        B_{r_2} \setminus B_{r_1}
    }
    (x)
    {\mathcal F}
    \quad
    \mbox{in } B_{r_2},
$$
where
$$
    \chi_{
        B_{r_2} \setminus B_{r_1}
    }
    (x)
    =
    \left\{
        \begin{aligned}
            &
            1,
            &
            &
            x \in B_{r_2} \setminus B_{r_1},
            \\
            &
            0,
            &
            &
            x \not\in B_{r_2} \setminus B_{r_1},
        \end{aligned}
    \right.
$$
is the characteristic function of the set $B_{r_2} \setminus B_{r_1}$.
At the same time, according to~\eqref{pl3.3.4}, we have
$
    \tilde v (x) \ge M (r_1) - \varepsilon > M (r_1) / 2
$
for all $x \in B_{r_2}$.
Thus, $\tilde v$ is a weak solution of the equation
$$
    \Delta_\varphi \tilde v
    =
    \rho (x, \tilde v)
    \quad
    \mbox{in } B_{r_2}
$$
and, applying Lemma~\ref{l3.1}, we arrive at the inequality
$
    \tilde u \le \tilde v
    \quad
    \mbox{a.e. in } B_{r_2}
$
or, in other words,
$$
    \tilde u (x)
    \le
    \tilde w (x)
    +
    \operatorname*{ess\,sup}\limits_{
        B_{r_2}
    }
    \,
    (\tilde u - \tilde w)
    -
    \varepsilon
$$
for almost all $x \in B_{r_2}$,
whence it follows that
$$
    \tilde u (x)
    -
    \tilde w (x)
    \le
    \operatorname*{ess\,sup}\limits_{
        B_{r_2}
    }
    \,
    (\tilde u - \tilde w)
    -
    \varepsilon
$$
for almost all $x \in B_{r_2}$.
This contradiction proves~\eqref{pl3.3.5}.

Since
$$
    \left.
        v
    \right|_{
        S_{r_1}
    }
    =
    \sup_{
        B_{r_1}
    }
    v
    \ge
    \operatorname*{ess\,sup}\limits_{
        B_{r_1}
    }
    \tilde u
    =
    M (r_1)
$$
in accordance with~\eqref{pl3.3.3}, equality~\eqref{pl3.3.5} yields
$$
    M (r_2) - M (r_1)
    \ge
    \left.
        v
    \right|_{
        S_{r_2}
    }
    -
    \left.
        v
    \right|_{
        S_{r_1}
    }
    =
    w (r_2).
$$
Combining this with~\eqref{pl3.3.2}, we complete the proof.
\end{proof}

\begin{lemma}\label{l3.4}
Let $r_* \le r_1 < r_2$ be real numbers  such that $r_2 \le \sigma
r_1$ and $M (r_2) \le 2 M (r_1)$. Then,
\begin{equation}
    \frac{
        M (r_2) - M (r_1)
    }{
        \varphi^{-1} (g (\tau))
    }
    \ge
    \frac{
        C
        \eta (r_2 - r_1)
    }{
        \varphi^{-1}
        \left(
            \frac{1}{f (r)}
        \right)
    }
    \label{l3.4.1}
\end{equation}
for all $r \in [r_2 / \sigma, \sigma r_1]$ and $\tau \in [M (r_1)/ 2, 2 M (r_2)]$.
\end{lemma}

\begin{proof}
Take some $r \in [r_2 / \sigma, \sigma r_1]$  and $\tau \in [M (r_1)
/ 2, 2 M (r_2)]$. From~\eqref{1.5}, it follows that
\begin{equation}
    F (x, t)
    \ge
    f (r)
    g (t)
    \label{pl3.4.1}
\end{equation}
for all
$x \in \Omega \cap B_{r_2} \setminus B_{r_1}$
and
$t \in [M (r_1) / 2, M (r_2)]$.
Putting $\tau_1 = \tau$ and $\tau_2 = t / \tau$ in~\eqref{1.8}, we have
$
    g (t)
    \ge
    {
        C g (\tau)
    }/{
        g (\tau / t)
    }.
$
Since the function $g$ is bounded from above on any compact subset
of the interval $(0, \infty)$ and, moreover,
$\tau / t \in [1 / 4, 8]$ for all $t \in [M (r_1) / 2, M (r_2)]$,
the last inequality yields
$
    g (t) \ge C g (\tau)
$
for all $t \in [M (r_1) / 2, M (r_2)]$.
Thus,~\eqref{pl3.4.1} implies the estimate
$
    {\mathcal F} \ge C f (r) g (\tau),
$
where ${\mathcal F}$ is given by~\eqref{l3.3.2}.
Combining this with formula~\eqref{l3.3.1} of Lemma~\ref{l3.3}, we obtain
\begin{equation}
    M (r_2) - M (r_1)
    \ge
    C
    (r_2 - r_1)
    \varphi^{-1}
    (
        (r_2 - r_1)
        f (r)
        g (\tau)
    ).
    \label{pl3.4.2}
\end{equation}
As $\varphi^{-1}$ is an almost semi-multiplicative function, one can assert that
$$
    \varphi^{-1}
    (
        (r_2 - r_1)
        f (r)
        g (\tau)
    )
    \ge
    \frac{
        C
        \varphi^{-1}
        (
            g (\tau)
        )
    }{
        \varphi^{-1}
        \left(
            \frac{
                1
            }{
                (r_2 - r_1)
                f (r)
            }
        \right)
    }
$$
and
$$
    \varphi^{-1}
    \left(
        \frac{
            1
        }{
            (r_2 - r_1)
            f (r)
        }
    \right)
    \le
    C
    \varphi^{-1}
    \left(
        \frac{
            1
        }{
            r_2 - r_1
        }
    \right)
    \varphi^{-1}
    \left(
        \frac{
            1
        }{
            f (r)
        }
    \right).
$$
Thus,~\eqref{pl3.4.2} leads to the inequality
$$
    M (r_2) - M (r_1)
    \ge
    \frac{
        C
        (r_2 - r_1)
        \varphi^{-1}
        (
            g (\tau)
        )
    }{
        \varphi^{-1}
        \left(
            \frac{
                1
            }{
                r_2 - r_1
            }
        \right)
        \varphi^{-1}
        \left(
            \frac{
                1
            }{
                f (r)
            }
        \right)
    }
$$
which is obviously equivalent to~\eqref{l3.4.1}.
\end{proof}

The following statement is simple but useful.

\begin{lemma}\label{l3.5}
Let $a > 0$ and $b > 0$ be real numbers. Then,
\begin{equation}
    \eta (a) b
    \ge
    C
    \eta
    \left(
        \frac{
            a
        }{
            \eta^{-1}
            \left(
                \frac{1}{b}
            \right)
        }
    \right).
    \label{l3.5.1}
\end{equation}
\end{lemma}

\begin{proof}
Taking into account the relation
$$
    \varphi^{-1}
    \left(
        \frac{1}{\alpha \beta}
    \right)
    \ge
    \frac{
        C
        \varphi^{-1}
        \left(
            \frac{1}{\alpha}
        \right)
    }{
        \varphi^{-1} (\beta)
    },
$$
we obtain
\begin{equation}
    \eta (\alpha \beta)
    =
    \frac{
        \alpha \beta
    }{
        \varphi^{-1}
        \left(
            \frac{1}{\alpha \beta}
        \right)
    }
    \le
    \frac{
        C
        \alpha
        \beta
        \varphi^{-1} (\beta)
    }{
        \varphi^{-1}
        \left(
            \frac{1}{\alpha}
        \right)
    }
    \label{pl3.5.1}
\end{equation}
for all real numbers $\alpha > 0$ and $\beta > 0$.
Let us denote
$
    h (t)
    =
    t
    \varphi^{-1} (t),
    \quad
    t > 0,
$
It can be verified that
$
    h (t)
    =
    \frac{
        1
    }{
        \eta
        \left(
            {1}/{t}
        \right)
    },
    \quad
    t > 0.
$
We also note that $h$ is an increasing  one-to-one mapping of the
set $(0, \infty)$ into itself with
$$
    h^{-1} (t)
    =
    \frac{
        1
    }{
        \eta^{-1}
        \left(
            {1}/{t}
        \right)
    },
    \quad
    t > 0.
$$
Formula~\eqref{pl3.5.1} can obviously be written in the form
$
    \eta (\alpha \beta) \le C \eta (\alpha) h (\beta),
$
whence, putting $\alpha = a$ and $\beta = h^{-1} (b)$,  we
immediately arrive at~\eqref{l3.5.1}.
\end{proof}

\begin{lemma}\label{l3.6}
Let $r_* \le r_1 < r_2$ be real numbers  such that $r_2 =
\sigma^{1/2} r_1$ and $M (r_2) \le 2 M (r_1)$. Then,
$$
    \int_{
        M (r_1)
    }^{
        M (r_2)
    }
    \frac{
        dt
    }{
        \varphi^{-1} (g (t))
    }
    \ge
    C
    \int_{
        r_1
    }^{
        r_2
    }
    \frac{
        \eta (r)
        dr
    }{
        r
        \varphi^{-1}
        \left(
            \frac{
                1
            }{
                f (r)
            }
        \right)
    }.
$$
\end{lemma}

\begin{proof}
Taking into account Lemma~\ref{l3.4} and the  fact that $\eta$ is an
increasing function satisfying relation~\eqref{1.10}, we obtain
$$
    \frac{
        M (r_2) - M (r_1)
    }{
        \varphi^{-1} (g (\tau))
    }
    \ge
    \frac{
        C
        \eta (r)
    }{
        \varphi^{-1}
        \left(
            \frac{1}{f (r)}
        \right)
    }
$$
for all $r \in [r_1, r_2]$ and $\tau \in [M (r_1), M (r_2)]$.
Combining this with the evident inequalities
$$
    \int_{
        M (r_1)
    }^{
        M (r_2)
    }
    \frac{
        dt
    }{
        \varphi^{-1} (g (t))
    }
    \ge
    \inf_{
        \tau \in [M (r_1), M (r_2)]
    }
    \frac{
        M (r_2) - M (r_1)
    }{
        \varphi^{-1} (g (\tau))
    }
$$
and
$$
    \frac{1}{2}
    \ln \sigma
    \sup_{
        r \in [r_1, r_2]
    }
    \frac{
        \eta (r)
    }{
        \varphi^{-1}
        \left(
            \frac{1}{f (r)}
        \right)
    }
    \ge
    \int_{
        r_1
    }^{
        r_2
    }
    \frac{
        \eta (r)
        dr
    }{
        r
        \varphi^{-1}
        \left(
            \frac{
                1
            }{
                f (r)
            }
        \right)
    },
$$
we complete the proof.
\end{proof}

\begin{lemma}\label{l3.7}
Let the conditions of Lemma~\ref{l3.6} be valid, then
$$
    \int_{
        M (r_1)
    }^{
        M (r_2)
    }
    \frac{
        dt
    }{
        \varphi^{-1} (g (t))
    }
    \ge
    C
    \int_{
        r_1
    }^{
        r_2
    }
    \eta
    \left(
        \frac{
            r
        }{
            \eta^{-1}
            \left(
                \varphi^{-1}
                \left(
                    \frac{1}{f (r)}
                \right)
            \right)
        }
    \right)
    \frac{dr}{r}.
$$
\end{lemma}

\begin{proof}
Applying Lemma~\ref{l3.5} with $a = r$ and
$$
    b
    =
    \frac{
        1
    }{
        \varphi^{-1}
        \left(
            \frac{
                1
            }{
                f (r)
            }
        \right)
    },
$$
we obtain
$$
    \frac{
        \eta (r)
    }{
        \varphi^{-1}
        \left(
            \frac{
                1
            }{
                f (r)
            }
        \right)
    }
    \ge
    C
    \eta
    \left(
        \frac{
            r
        }{
            \eta^{-1}
            \left(
                \varphi^{-1}
                \left(
                    \frac{1}{f (r)}
                \right)
            \right)
        }
    \right)
$$
for all $r \in [r_1, r_2]$, whence it follows that
$$
    \int_{
        r_1
    }^{
        r_2
    }
    \frac{
        \eta (r)
        dr
    }{
        r
        \varphi^{-1}
        \left(
            \frac{
                1
            }{
                f (r)
            }
        \right)
    }
    \ge
    C
    \int_{
        r_1
    }^{
        r_2
    }
    \eta
    \left(
        \frac{
            r
        }{
            \eta^{-1}
            \left(
                \varphi^{-1}
                \left(
                    \frac{1}{f (r)}
                \right)
            \right)
        }
    \right)
    \frac{dr}{r}.
$$
Thus, to complete the proof, it remains to use Lemma~\ref{l3.6}.
\end{proof}

\begin{lemma}\label{l3.8}
Let $r_* \le r_1 < r_2$ be real numbers such that $r_2 =
\sigma^{1/2} r_1$ and $M (r_2) \ge 2 M (r_1)$. Then,
\begin{equation}
    \int_{
        M (r_1)
    }^{
        M (r_2)
    }
    \eta^{-1}
    \left(
        \frac{
            t
        }{
            \varphi^{-1} (g (t))
        }
    \right)
    \frac{dt}{t}
    \ge
    C
    \int_{
        r_1
    }^{
        r_2
    }
    \Phi (r)
    \frac{dr}{r},
    \label{l3.8.1}
\end{equation}
where
\begin{equation}
    \Phi (r)
    =
    \sup_{
        s \in (\sigma^{- 1/2} r, \, r \sigma^{1/2})
    }
    \,
    \frac{
        s
    }{
        \eta^{-1}
        \left(
            \varphi^{-1}
            \left(
                \frac{1}{f (s)}
            \right)
        \right)
    }.
    \label{l3.8.2}
\end{equation}
\end{lemma}

\begin{proof}
Consider a finite sequence of real  numbers $s_1 > s_2 > \ldots >
s_k$ constructed as follows. Let us take $s_1 = r_2$. Assume further
that $s_i$ is already known for some positive integer $i$. If $s_i =
r_1$, we put $k  = i$ and stop; otherwise we take
$
    s_{i + 1}
    =
    \inf
    \{
        r \in [r_1, s_i]
        :
        M (s_i)
        \le
        2
        M (r)
    \}.
$
It is obvious that
$$
    M (s_i) \le 2 M (s_{i+1} + 0),
    \quad
    i = 1, 2, \ldots, k - 1.
$$
Since $M (r) = M (r - 0)$ for all $r \in (r_1, r_2]$, we also have
\begin{equation}
    2 M (s_{i+1}) \le M (s_i),
    \quad
    i = 1, 2, \ldots, k - 2.
    \label{pl3.8.1}
\end{equation}

By Lemma~\ref{l3.4},
$$
    \frac{
        M (s_i) - M (s_{i+1} + 0)
    }{
        \varphi^{-1} (g (\tau))
    }
    \ge
    \frac{
        C
        \eta (s_i - s_{i+1})
    }{
        \varphi^{-1}
        \left(
            \frac{1}{f (r)}
        \right)
    }
$$
for all $r \in (s_i / \sigma, \sigma s_{i+1})$ and $\tau \in (M (s_{i+1} + 0) / 2, 2 M (s_i))$,
$s = 1, 2, \ldots, k - 1$.
At the same time, taking into account Lemma~\ref{l3.5}, we obtain
$$
    \frac{
        \eta (s_i - s_{i+1})
    }{
        \varphi^{-1}
        \left(
            \frac{
                1
            }{
                f (r)
            }
        \right)
    }
    \ge
    C
    \eta
    \left(
        \frac{
            s_i - s_{i+1}
        }{
            \eta^{-1}
            \left(
                \varphi^{-1}
                \left(
                    \frac{1}{f (r)}
                \right)
            \right)
        }
    \right);
$$
therefore, one can assert that
$$
    \frac{
        M (s_i) - M (s_{i+1} + 0)
    }{
        \varphi^{-1} (g (\tau))
    }
    \ge
    C
    \eta
    \left(
        \frac{
            s_i - s_{i+1}
        }{
            \eta^{-1}
            \left(
                \varphi^{-1}
                \left(
                    \frac{1}{f (r)}
                \right)
            \right)
        }
    \right),
$$
for all $r \in (s_i / \sigma, \sigma s_{i+1})$ and $\tau \in (M
(s_{i+1} + 0) / 2, 2 M (s_i))$, $s = 1, 2, \ldots, k - 1$.
By~\eqref{1.11} and the fact that $\eta^{-1}$  is an increasing
function, this implies the inequality
$$
    \eta^{-1}
    \left(
        \frac{
            M (s_i) - M (s_{i+1} + 0)
        }{
            \varphi^{-1} (g (\tau))
        }
    \right)
    \ge
    \frac{
        C (s_i - s_{i+1})
    }{
        \eta^{-1}
        \left(
            \varphi^{-1}
            \left(
                \frac{1}{f (r)}
            \right)
        \right)
    }
$$
for all $r \in (s_i / \sigma, \sigma s_{i+1})$ and  $\tau \in (M
(s_{i+1} + 0) / 2, 2 M (s_i))$, $s = 1, 2, \ldots, k - 1$, whence it
follows that
\begin{equation}
    \eta^{-1}
    \left(
        \frac{
            M (s_i)
        }{
            2 \varphi^{-1} (g (\tau))
        }
    \right)
    \ge
    C
    (s_i - s_{i+1})
    \frac{
        \Phi (r)
    }{
        r
    }
    \label{pl3.8.2}
\end{equation}
for all $r \in (s_{i+1}, s_i)$ and $\tau \in (M (s_{i+1} + 0) / 2, 2 M (s_i))$,
$s = 1, 2, \ldots, k - 1$.

Combining~\eqref{pl3.8.2} with the evident inequalities
$$
    \int_{
        M (s_i) / 2
    }^{
        M (s_i)
    }
    \eta^{-1}
    \left(
        \frac{
            t
        }{
            \varphi^{-1} (g (t))
        }
    \right)
    \frac{dt}{t}
    \ge
    \ln 2
    \inf_{
        \tau \in (M (s_i) / 2, M (s_i))
    }
    \eta^{-1}
    \left(
        \frac{
            M (s_i)
        }{
            2 \varphi^{-1} (g (\tau))
        }
    \right)
$$
and
\begin{equation}
    (s_i - s_{i+1})
    \sup_{
        r \in (s_{i+1}, s_i)
    }
    \frac{
        \Phi (r)
    }{
        r
    }
    \ge
    \int_{
        s_{i+1}
    }^{
        s_i
    }
    \Phi (r)
    \frac{dr}{r},
    \label{pl3.8.3}
\end{equation}
we arrive at the estimate
$$
    \int_{
        M (s_i) / 2
    }^{
        M (s_i)
    }
    \eta^{-1}
    \left(
        \frac{
            t
        }{
            \varphi^{-1} (g (t))
        }
    \right)
    \frac{dt}{t}
    \ge
    C
    \int_{
        s_{i+1}
    }^{
        s_i
    }
    \Phi (r)
    \frac{dr}{r}
$$
which, in view of~\eqref{pl3.8.1}, yields
$$
    \int_{
        M (s_{i+1})
    }^{
        M (s_i)
    }
    \eta^{-1}
    \left(
        \frac{
            t
        }{
            \varphi^{-1} (g (t))
        }
    \right)
    \frac{dt}{t}
    \ge
    C
    \int_{
        s_{i+1}
    }^{
        s_i
    }
    \Phi (r)
    \frac{dr}{r},
    \quad
    i = 1, 2, \ldots, k - 2.
$$
Summing the last expression over all $1 \le i \le k - 2$, we obtain
\begin{equation}
    \int_{
        M (s_{k-1})
    }^{
        M (r_2)
    }
    \eta^{-1}
    \left(
        \frac{
            t
        }{
            \varphi^{-1} (g (t))
        }
    \right)
    \frac{dt}{t}
    \ge
    C
    \int_{
        s_{k-1}
    }^{
        r_2
    }
    \Phi (r)
    \frac{dr}{r}.
    \label{pl3.8.4}
\end{equation}

Let us further take a real number $m \ge M (r_1)$ such that
$M (r_1 + 0) / 2 \le m \le M (r_1 + 0)$
and
$M (s_{k-1}) \le 2 m \le M (r_2)$.
Such a real number $m$ obviously exists.

It can be seen that
$$
    \int_{
        m
    }^{
        2 m
    }
    \eta^{-1}
    \left(
        \frac{
            t
        }{
            \varphi^{-1} (g (t))
        }
    \right)
    \frac{dt}{t}
    \ge
    \ln 2
    \inf_{
        \tau \in (m, 2m)
    }
    \eta^{-1}
    \left(
        \frac{
            M (s_{k-1})
        }{
            2 \varphi^{-1} (g (\tau))
        }
    \right);
$$
therefore, taking into  account~\eqref{pl3.8.2} and~\eqref{pl3.8.3}
with $s = k - 1$, we have
$$
    \int_{
        m
    }^{
        2 m
    }
    \eta^{-1}
    \left(
        \frac{
            t
        }{
            \varphi^{-1} (g (t))
        }
    \right)
    \frac{dt}{t}
    \ge
    C
    \int_{
        r_1
    }^{
        s_{k-1}
    }
    \Phi (r)
    \frac{dr}{r}.
$$
Finally, summing the last inequality  and~\eqref{pl3.8.4}, we arrive
at~\eqref{l3.8.1}.
\end{proof}

We need the following known result.

\begin{lemma}\label{l3.9}
Let $\mu > 1$ and $\nu > 1$ be some real numbers and $H : [0,
\infty) \to [0, \infty)$ be a convex function with $\eta (0) = 0$.
Then,
$$
    H
    \left(
        \int_{t_1}^{t_2}
        \Phi (t)
        \frac{dt}{t}
    \right)
    \ge
    A
    \int_{t_1}^{t_2}
    H (B \Psi (t))
    \frac{dt}{t}
$$
for any measurable function $\Phi : (0, \infty) \to [0, \infty)$
and positive real numbers $t_1$ and $t_2$ such that $\mu t_1 \le t_2$, where
$
    \Psi (t)
    =
    \operatorname*{ess\,inf}_{
        (t / \nu, \nu t)
        \cap
        (t_1, t_2)
    }
    \Phi
$
and the constants $A > 0$ and $B > 0$ depend only on $\mu$ and $\nu$.
\end{lemma}

Lemma~\ref{l3.9} is proved in~\cite[Lemma~3.1]{meSM}.

\begin{proof}[Proof of Theorem~\ref{t2.1}]
We put $r_i = \sigma^{i/2} r_*$, $i = 0,1,2,\ldots$.
In so doing, let $\Xi_1$ be the set of positive integers $i$ such that
$2 M (r_{i-1}) \ge M (r_i)$ and $\Xi_2$ be the set of all other positive integers.

By Lemma~\ref{l3.7}, we obtain
$$
    \int_{
        M (r_{i-1})
    }^{
        M (r_i)
    }
    \frac{
        dt
    }{
        \varphi^{-1} (g (t))
    }
    \ge
    C
    \int_{
        r_{i-1}
    }^{
        r_i
    }
    \eta
    \left(
        \frac{
            r
        }{
            \eta^{-1}
            \left(
                \varphi^{-1}
                \left(
                    \frac{1}{f (r)}
                \right)
            \right)
        }
    \right)
    \frac{dr}{r}
$$
for all $i \in \Xi_1$, whence it follows that
\begin{equation}
    \int_{
        M (r_*)
    }^\infty
    \frac{
        dt
    }{
        \varphi^{-1} (g (t))
    }
    \ge
    C
    \sum_{i \in \Xi_1}
    \int_{
        r_{i-1}
    }^{
        r_i
    }
    \eta
    \left(
        \frac{
            r
        }{
            \eta^{-1}
            \left(
                \varphi^{-1}
                \left(
                    \frac{1}{f (r)}
                \right)
            \right)
        }
    \right)
    \frac{dr}{r}.
    \label{pt2.1.1}
\end{equation}
Since $\varphi^{-1}$ and $g$ are almost  semi-multiplicative
functions and, moreover,~\eqref{1.10} and~\eqref{1.11} hold, we have
$$
    \inf_{
        \tau \in (t / 2, 2 t)
    }
    \eta^{-1}
    \left(
        \frac{
            \tau
        }{
            \varphi^{-1} (g (\tau))
        }
    \right)
    \ge
    C
    \eta^{-1}
    \left(
        \frac{
            t
        }{
            \varphi^{-1} (g (t))
        }
    \right)
$$
for all real numbers $t > M (r_*)$. Consequently,  Lemma~\ref{l3.9}
implies that
$$
    \eta
    \left(
        \int_{
            M (r_*)
        }^\infty
        \eta^{-1}
        \left(
            \frac{
                t
            }{
                \varphi^{-1} (g (t))
            }
        \right)
        \frac{dt}{t}
    \right)
    \ge
    C
    \int_{
        M (r_*)
    }^\infty
    \frac{
        dt
    }{
        \varphi^{-1} (g (t))
    }.
$$
Combining this with~\eqref{pt2.1.1}, we derive
\begin{equation}
    \eta
    \left(
        \int_{
            M (r_*)
        }^\infty
        \eta^{-1}
        \left(
            \frac{
                t
            }{
                \varphi^{-1} (g (t))
            }
        \right)
        \frac{dt}{t}
    \right)
    \ge
    C
    \sum_{i \in \Xi_1}
    \int_{
        r_{i-1}
    }^{
        r_i
    }
    \eta
    \left(
        \frac{
            r
        }{
            \eta^{-1}
            \left(
                \varphi^{-1}
                \left(
                    \frac{1}{f (r)}
                \right)
            \right)
        }
    \right)
    \frac{dr}{r}.
    \label{pt2.1.2}
\end{equation}

In its turn, from Lemma~\ref{l3.8}, it follows that
\begin{equation}
    \int_{
        M (r_{i-1})
    }^{
        M (r_i)
    }
    \eta^{-1}
    \left(
        \frac{
            t
        }{
            \varphi^{-1} (g (t))
        }
    \right)
    \frac{dt}{t}
    \ge
    C
    \int_{
        r_{i-1}
    }^{
        r_i
    }
    \Phi (r)
    \frac{dr}{r}
    \label{pt2.1.3}
\end{equation}
for all $i \in \Xi_2$, where $\Phi$ is given by~\eqref{l3.8.2}.
This immediately yields
\begin{equation}
    \int_{
        M (r_*)
    }^\infty
    \eta^{-1}
    \left(
        \frac{
            t
        }{
            \varphi^{-1} (g (t))
        }
    \right)
    \frac{dt}{t}
    \ge
    C
    \sum_{i \in \Xi_2}
    \int_{
        r_{i-1}
    }^{
        r_i
    }
    \Phi (r)
    \frac{dr}{r}.
    \label{pt2.1.4}
\end{equation}
At the same time, we have
$$
    \inf_{
        s \in (\sigma^{- 1/2} r, \, r \sigma^{1/2})
    }
    \Phi (s)
    \ge
    \frac{
        r
    }{
        \eta^{-1}
        \left(
            \varphi^{-1}
            \left(
                \frac{1}{f (r)}
            \right)
        \right)
    }
$$
for all $r \in (r_*, \infty)$. Hence,  Lemma~\ref{l3.9} allows us to
assert that
$$
    \eta
    \left(
        \sum_{i \in \Xi_2}
        \int_{
            r_{i-1}
        }^{
            r_i
        }
        \Phi (r)
        \frac{dr}{r}
    \right)
    \ge
    \sum_{i \in \Xi_2}
    \eta
    \left(
        \int_{
            r_{i-1}
        }^{
            r_i
        }
        \Phi (r)
        \frac{dr}{r}
    \right)
    \ge
    C
    \sum_{i \in \Xi_2}
    \int_{
        r_{i-1}
    }^{
        r_i
    }
    \eta
    \left(
        \frac{
            r
        }{
            \eta^{-1}
            \left(
                \varphi^{-1}
                \left(
                    \frac{1}{f (r)}
                \right)
            \right)
        }
    \right)
    \frac{dr}{r}.
$$
Combining this with~\eqref{pt2.1.4}, we obtain
$$
    \eta
    \left(
        \int_{
            M (r_*)
        }^\infty
        \eta^{-1}
        \left(
            \frac{
                t
            }{
                \varphi^{-1} (g (t))
            }
        \right)
        \frac{dt}{t}
    \right)
    \ge
    C
    \sum_{i \in \Xi_2}
    \int_{
        r_{i-1}
    }^{
        r_i
    }
    \eta
    \left(
        \frac{
            r
        }{
            \eta^{-1}
            \left(
                \varphi^{-1}
                \left(
                    \frac{1}{f (r)}
                \right)
            \right)
        }
    \right)
    \frac{dr}{r}.
$$
Summing the last formula and~\eqref{pt2.1.2}, one can conclude that
$$
    \eta
    \left(
        \int_{
            M (r_*)
        }^\infty
        \eta^{-1}
        \left(
            \frac{
                t
            }{
                \varphi^{-1} (g (t))
            }
        \right)
        \frac{dt}{t}
    \right)
    \ge
    C
    \int_{
        r_*
    }^\infty
    \eta
    \left(
        \frac{
            r
        }{
            \eta^{-1}
            \left(
                \varphi^{-1}
                \left(
                    \frac{1}{f (r)}
                \right)
            \right)
        }
    \right)
    \frac{dr}{r}.
$$
This obviously contradicts conditions~\eqref{t2.1.1}
and~\eqref{t2.1.2}. Since all the previous statements have been
proven under the assumption that $M (r_*) > 0$, our contradiction
means that $u = 0$ a.e. in $\Omega$.
\end{proof}

\begin{proof}[Proof of Theorem~\ref{t2.2}]
Let the real numbers $r_i$, $i = 0,1,2,\ldots$,  and the sets
$\Xi_1$ and $\Xi_2$ be the same as in the proof of
Theorem~\ref{t2.1}. By Lemma~\ref{l3.6}, we have
$$
    \int_{
        M (r_{i-1})
    }^{
        M (r_i)
    }
    \frac{
        dt
    }{
        \varphi^{-1} (g (t))
    }
    \ge
    C
    \int_{
        r_{i-1}
    }^{
        r_i
    }
    \frac{
        \eta (r)
        dr
    }{
        r
        \varphi^{-1}
        \left(
            \frac{
                1
            }{
                f (r)
            }
        \right)
    }
$$
for all $i \in \Xi_1$. In its turn, Lemma~\ref{l3.8} implies
estimate~\eqref{pt2.1.3} for all $i \in \Xi_2$, where $\Phi$ is
given by~\eqref{l3.8.2}, whence in accordance with the evident
inequality
$$
    \Phi (r)
    \ge
    \frac{
        r
    }{
        \eta^{-1}
        \left(
            \varphi^{-1}
            \left(
                \frac{1}{f (r)}
            \right)
        \right)
    }
$$
it follows that
$$
    \int_{
        M (r_{i-1})
    }^{
        M (r_i)
    }
    \eta^{-1}
    \left(
        \frac{
            t
        }{
            \varphi^{-1} (g (t))
        }
    \right)
    \frac{dt}{t}
    \ge
    C
    \int_{
        r_{i-1}
    }^{
        r_i
    }
    \frac{
        dr
    }{
        \eta^{-1}
        \left(
            \varphi^{-1}
            \left(
                \frac{1}{f (r)}
            \right)
        \right)
    }
$$
for all $i \in \Xi_2$.
Thus, one can assert that
\begin{align*}
    &
    \int_{
        M (r_{i-1})
    }^{
        M (r_i)
    }
    \frac{
        dt
    }{
        \varphi^{-1} (g (t))
    }
    +
    \int_{
        M (r_{i-1})
    }^{
        M (r_i)
    }
    \eta^{-1}
    \left(
        \frac{
            t
        }{
            \varphi^{-1} (g (t))
        }
    \right)
    \frac{dt}{t}
    \\
    &
    \qquad
    \ge
    C
    \int_{
        r_{i-1}
    }^{
        r_i
    }
    \min
    \left\{
        \frac{
            \eta (r)
        }{
            r
            \varphi^{-1}
            \left(
                \frac{
                    1
                }{
                    f (r)
                }
            \right)
        },
        \frac{
            1
        }{
            \eta^{-1}
            \left(
                \varphi^{-1}
                \left(
                    \frac{
                        1
                    }{
                        f (r)
                    }
                \right)
            \right)
        }
    \right\}
    dr
\end{align*}
for all $i = 1,2,\ldots$.
Summing the last estimate over all positive integers $i$, we obtain
\begin{align*}
    &
    \int_{
        M (r_*)
    }^\infty
    \frac{
        dt
    }{
        \varphi^{-1} (g (t))
    }
    +
    \int_{
        M (r_*)
    }^\infty
    \eta^{-1}
    \left(
        \frac{
            t
        }{
            \varphi^{-1} (g (t))
        }
    \right)
    \frac{dt}{t}
    \\
    &
    \qquad
    \ge
    C
    \int_{
        r_*
    }^\infty
    \min
    \left\{
        \frac{
            \eta (r)
        }{
            r
            \varphi^{-1}
            \left(
                \frac{
                    1
                }{
                    f (r)
                }
            \right)
        },
        \frac{
            1
        }{
            \eta^{-1}
            \left(
                \varphi^{-1}
                \left(
                    \frac{
                        1
                    }{
                        f (r)
                    }
                \right)
            \right)
        }
    \right\}
    dr,
\end{align*}
which contradicts conditions~\eqref{t2.1.1},  \eqref{t2.2.1},
and~\eqref{t2.2.2}. Since the reason for the contradiction is our
assumption that $M (r_*) > 0$, we have $u = 0$ a.e. in $\Omega$.
\end{proof}


\end{document}